\documentclass[10pt]{article}

\usepackage{color}
\usepackage{amstext}
\usepackage{amssymb}
\usepackage{amsmath}
\usepackage{amsthm}
\usepackage{fullpage}
\usepackage[dvips]{graphicx}        
\graphicspath{{./pict/}}            

\usepackage[cp1251]{inputenc}
\usepackage[russian,english]{babel}

\theoremstyle{plain}
\newtheorem{theorem}{Theorem}[section]
\newtheorem{lemma}[theorem]{Lemma}

\theoremstyle{definition}

\newtheorem{remark}[theorem]{Remark}

\renewcommand{\Im}{\operatorname{Im}}
\renewcommand{\Re}{\operatorname{Re}}

\def\sgn{\mathop{\rm sign}\nolimits}

\def\sgn{\mathop{\rm sign}\nolimits}

\def\sgn{\mathop{\rm sign}\nolimits}

\begin{document}

\title{Direct and inverse spectral problems for a class
of non-selfadjoint band matrices}

\author{Natalia Bebiano and Mikhail Tyaglov}



\maketitle

 \maketitle

\abstract{The spectral properties of a class of  band matrices are
investigated. The reconstruction of matrices of this special class
from given spectral data is also studied.
Necessary and sufficient conditions for that reconstruction are
found. The obtained results extend some results on the direct and
inverse spectral problems for periodic Jacobi matrices and for some
non-self-adjoint tridiagonal matrices.

\setcounter{equation}{0}

AMS: {65 F15, 15A18}



\section{Introduction}

Inverse eigenvalue problems arise in mathematics as well as in many
areas of engineering and science such as chemistry, geology, physics
etc. Often the mathematical model describing a certain physical
system involves matrices whose spectral data allow the prediction of
the behavior of the system. Determining the spectra of those
matrices is the so-called {\it direct problem}, while the {\it
inverse problem} consists in the reconstruction of the matrices from
the knowledge of the behavior of the system, frequently expressed by
spectral data.

Inverse eigenvalue problems, in general, and for structured matrices,
in particular,  have attracted attention of many researchers,
some of them motivated by the numerous applications of this
scientific area (see e.g. \cite{akhiezer}, \cite{G}).
The mathematical background employed in those investigations may
involve rather sophisticated techniques such as algebraic curves,
functional analysis, matrix theory, etc. (see \cite{vM}, \cite{BG1},
\cite{BG2}, \cite{AB}, \cite{Teschl} and the references therein).

Inverse eigenvalue problems  for band matrices have been actively investigated,
e.g. see \cite{BG1} and their references. The inverse spectral problem for a
{\it  periodic Jacobi matrix}, that is, a real symmetric matrix of the form

\begin{equation}\label{main.matrix}
L_n=\begin{pmatrix}
    a_1 & b_1 &  0 &\dots&   0   & b_n \\
    b_1 & a_2 &b_2 &\dots&   0   & 0 \\
     0  &b_2 & a_3 &\dots&   0   & 0 \\
    \vdots&\vdots&\vdots&\ddots&\vdots&\vdots\\
     0  &  0  &  0  &\dots&a_{n-1}& b_{n-1}\\
     b_n  &  0  &  0  &\dots&b_{n-1}&a_n\\
\end{pmatrix}, \qquad b_i>0,
\end{equation}
deserved the attention of researchers, see \cite{BG2,F,AB,XJO} and
their references. These matrices appear in studies of the periodic
Toda lattice, inverse eigenvalue problems for Sturm-Liouville
equations and Hill's equation \cite{BG1,F}. If $b_n=0$, the the matrices $L_n$
of the form~\eqref{main.matrix} reduce to tridiagonal symmetric matrices called the {\it Jacobi
matrices}. The Jacobi matrices motivated intensive study  as an useful tool in the
investigation of orthogonal polynomials, in the theory of continued
fractions, and in numerical analysis \cite{W,B1,B}. Namely, the inverse
problems for {Jacobi matrices} have been an intensive topic of
research since the seminal papers by Hochstadt and Hald
\cite{hochstadt,hald} in the seventies of the last century.

In the present work, we study spectral properties of  complex
matrices of the form
\begin{equation}\label{main.matrix.2}
J_n=\begin{pmatrix}
c_1 & b_1 &  0 &\dots&   0   & \overline{b}_n \\
\overline{b}_1 & c_2 &b_2 &\dots&   0   & 0 \\
0  &\overline{b}_2 & c_3 &\dots&   0   & 0 \\
\vdots&\vdots&\vdots&\ddots&\vdots&\vdots\\
0  &  0  &  0  &\dots&c_{n-1}& b_{n-1}\\
b_n  &  0  &  0  &\dots&\overline{b}_{n-1}&a_n\\
\end{pmatrix},
\end{equation}
where $b_1,\ldots,b_{n-1},b_n\in\mathbb{C}\backslash\mathbb{R},
c_1,\ldots,c_{n-1}\in\mathbb{R}$, and $a_n\in\mathbb{C}$, and solve
the direct problem for such matrices. (Here $\overline{z}$ means the
complex conjugate of $z$.) The  matrices of the
form~\eqref{main.matrix.2} constitute the class $\mathcal{J}_n$.

We also solve the inverse spectral problem for matrices from the subclass
$\widehat{\mathcal{J}}_n$ of the class $\mathcal{J}_n$. This class consists of
the matrices of the form
\begin{equation}\label{main.matrix.2.subclasss}
\widehat{J}_n=\begin{pmatrix}
\widehat{c}_1 & \widehat{b}_1 &  0 &\dots&   0   & \overline{\widehat{b}}_n \\
\widehat{b}_1 & \widehat{c}_2 &\widehat{b}_2 &\dots&   0   & 0 \\
0  &\widehat{b}_2 & \widehat{c}_3 &\dots&   0   & 0 \\
\vdots&\vdots&\vdots&\ddots&\vdots&\vdots\\
0  &  0  &  0  &\dots&\widehat{c}_{n-1}& \widehat{b}_{n-1}\\
\widehat{b}_n  &  0  &  0  &\dots&\widehat{b}_{n-1}&\widehat{a}_n\\
\end{pmatrix},
\end{equation}
where $\widehat{b}_1,\ldots,\widehat{b}_{n-1}, \widehat{c}_1,\ldots,\widehat{c}_{n-1}\in\mathbb{R}$,
and $\widehat{a}_n,\widehat{b}_n\in\mathbb{C}$, $\widehat{b}_n\neq0$.

Note that in~\cite{Arlinskii_Tsekhanovskii} (see
also~\cite{Arlinskii_Tsekhanovskii2}), Arlinskii and Tsekhanovskii
considered the matrices of the form~\eqref{main.matrix.2.subclasss}
with $b_n=0$, $a_n\in\mathbb{C}\backslash\mathbb{R}$, and solved the
direct and inverse eigenvalue  problems for
those matrices. In~\cite{XJO}, the direct and inverse spectral
problems for the matrices the form~\eqref{main.matrix.2.subclasss}
with $b_n\in\mathbb{R}\backslash\{0\}$ and $a_n\in\mathbb{R}$ (that is, for the matrices of the form~\eqref{main.matrix})
were solved, and necessary and sufficient conditions for solvability of the inverse problem
were found.

Recall that in~\cite{XJO}, it was established that the necessary and sufficient conditions for the inverse spectral problem
for the matrices of the form~\eqref{main.matrix} to be solvable are
\begin{equation}\label{Chinese_inequality}
\prod\limits_{j=1}^{n}|\mu_k-\lambda_j|\geqslant(-1)^{n-k-1}\beta,\qquad k=1,\ldots,n-1.
\end{equation}
Here $\beta=b_1\cdots b_n$ and the sets $\{\lambda_j\}_{j=1}^n$ and
$\{\mu_k\}_{k=1}^{n-1}$ are the spectra of the
matrix~\eqref{main.matrix} and its  $(n-1)\times(n-1)$ leading
principal submatrix, respectively.

In~\cite{Arlinskii_Tsekhanovskii} (see also~\cite{Arlinskii_Tsekhanovskii2}), it was established
that any matrix of the form~\eqref{main.matrix.2.subclasss} with $\widehat{b}_n=0$ and $\Im a_n>0$
has its eigenvalues in the open half-plane of the complex plane. The inverse spectral problem
of such matrices was also solved in~\cite{Arlinskii_Tsekhanovskii}.

In the present work, we extend the results of the work~\cite{XJO} and find necessary and sufficient conditions
for solvability of the inverse spectral problem for the matrices from the class $\widehat{\mathcal{J}}_n$.
Those conditions are also necessary for the matrices from the class $\mathcal{J}_n$ and generalize the
inequality~\eqref{Chinese_inequality}. Also we extended the results of the paper~\cite{Arlinskii_Tsekhanovskii}
to the matrices from the class $\mathcal{J}_n$.

The paper is organized as follows. In Section~\ref{section:preliminaries}, we
survey some general properties of the real symmetric tridiagonal matrices.
Section~\ref{section:spectrum.odd.even} is devoted to the study of
spectral properties of matrices from the class $\mathcal{J}_n$. We study
the eigenvalue location for real and nonreal numbers $\beta=b_1\cdots b_n$
and $a_n$, and find a necessary condition for the spectra of those matrices.
In Section~\ref{section:inverse.problems.even.odd}, we show that for any
matrix from the class $\mathcal{J}_n$, there exists a matrix
from the class $\widehat{\mathcal{J}}_n$ with the same spectral data.
In this section, we solve the inverse spectral problem for the matrices
from the class $\widehat{\mathcal{J}}_n$, and establish that the necessary
condition obtained in Section~\ref{section:spectrum.odd.even} is also sufficient
for the inverse spectral problem to be solvable.

In Sections~\ref{section:spectrum.odd.even}
and~\ref{section:inverse.problems.even.odd},
we follow the approach developed in~\cite{XJO}.
However, we simplify the substantiation of their
technique and extend it to a more wide class
of matrices.

\setcounter{equation}{0}
\section{Preliminaries}\label{section:preliminaries}

Let us denote by $J_{k}$, $k=1,\ldots,n-1$, the $k$th principal
submatrix of $J_n$, and consider the matrix~$J_{n-1}$
\begin{equation}\label{main.submatrix.1.nonreal}
J_{n-1}=\begin{pmatrix}
c_1 & b_1 &  0 &\dots&   0   & 0 \\
\overline{b}_1 & c_2 &b_2 &\dots&   0   & 0 \\
0  &\overline{b}_2 & c_3 &\dots&   0   & 0 \\
\vdots&\vdots&\vdots&\ddots&\vdots&\vdots\\
0  &  0  &  0  &\dots&c_{n-2}& b_{n-2}\\
0  &  0  &  0  &\dots&\overline{b}_{n-2}&c_{n-1}\\
\end{pmatrix}.
\end{equation}

This matrix is self-adjoint, so its characteristic polynomial
\begin{equation*}\label{main.submatrix.1.char.poly}
\chi_{n-1}(\lambda)=\det(\lambda I_{n-1}-J_{n-1})
\end{equation*}
has only real zeroes, the eigenvalues of the matrix $J_{n-1}$. Here
$I_{n-1}$ is the identity matrix of size $n-1$. Note that the
spectrum of $J_{n-1}$ coincides with the spectrum of the
matrix~(cf.\cite{KreinGantmacher})
\begin{equation}\label{main.submatrix.1.real}
\widetilde{J}_{n-1}=\begin{pmatrix}
c_1 & |b_1| &  0 &\dots&   0   & 0 \\
|b_1| & c_2 &|b_2| &\dots&   0   & 0 \\
0  &|b_2| & c_3 &\dots&   0   & 0 \\
\vdots&\vdots&\vdots&\ddots&\vdots&\vdots\\
0  &  0  &  0  &\dots&c_{n-2}& |b_{n-2}|\\
0  &  0  &  0  &\dots&|b_{n-2}|&c_{n-1}\\
\end{pmatrix}.
\end{equation}
But this matrix is  a real Jacobi matrix so its
eigenvalues are simple. Thus we obtain that the spectrum of the
matrix $J_{n-1}$, $\sigma(J_{n-1})=\{\mu_1,\ldots,\mu_{n-1}\}$, is
real and simple~(e.g., see \cite{KreinGantmacher}):
\begin{equation*}
\mu_1<\mu_2<\cdots<\mu_{n-1},
\end{equation*}
so
\begin{equation}\label{main.submatrix.2.spectrum.signs}
\sgn\left(\chi'_{n-1}(\mu_k)\right)=(-1)^{n-k-1},
\end{equation}
where $\chi'_{n-1}(\lambda)$ is the derivative of the polynomial $\chi_{n-1}(\lambda)$

By $\mathbf{u}_{k}=(u_{k1},\ldots,u_{k,n-1})^{T}\in\mathbb{C}^{n-1}$ we
denote the eigenvector of $J_{n-1}$ corresponding to the eigenvalue
$\mu_k$, $k=1,\ldots,n-1$ such that
\begin{equation}\label{orthonormality}
\overline{\mathbf{u}}_{j}^{T}\mathbf{u}_{k}=\delta_{jk},
\end{equation}
where $\delta_{jk}$ is the Kronecker symbol. Then the {\it resolvent} of
the matrix $J_{n-1}$ has the form
\begin{equation}\label{submatrix.resolvent}
(\lambda I_{n-1}-J_{n-1})^{-1}=
\sum\limits_{k=1}^{n-1}\dfrac{\mathbf{u}_k\overline{\mathbf{u}}_k^{T}}{\lambda-\mu_k}.
\end{equation}

It is easy to see that
\begin{equation*}
\mathbf{e}_{1}^{T}(\lambda I_{n-1}-J_{n-1})^{-1}\textbf{e}_{n-1}=\dfrac{b_1b_2\cdots b_{n-2}}{\chi_{n-1}(\lambda)},
\end{equation*}
where $\textbf{e}_1=(1,0,\ldots,0,0)^{T}\in\mathbb{C}^{n-1}$ and $\textbf{e}_{n-1}= (0,0,\ldots,0,1)^{T}\in\mathbb{C}^{n-1}$.
The formula~\eqref{submatrix.resolvent} gives us
\begin{equation*}
\dfrac{b_1\cdots b_{n-2}}{\chi_{n-1}(\lambda)}=\sum\limits_{k=1}^{n-1}\dfrac{u_{k1}\overline{u}_{k,n-1}}{\lambda-\mu_k},
\end{equation*}
so we obtain
\begin{equation}\label{formula.vectors}
b_nu_{k1}b_{n-1}\overline{u}_{k,n-1}=\dfrac{\beta}{\chi'_{n-1}(\mu_k)},\quad k=1,\ldots,n-1,
\end{equation}
where
\begin{equation}\label{beta}
\beta=b_1b_2\cdots b_n\neq0.
\end{equation}
In particular, $u_{k1}\neq0$ and $u_{k,n-1}\neq0$.

The formula~\eqref{formula.vectors} implies that
\begin{equation}\label{formula.vectors.abs}
|b_nu_{k1}|^2|b_{n-1}\overline{u}_{k,n-1}|^2=\dfrac{|\beta|^2}{[\chi'_{n-1}(\mu_k)]^2},\quad
k=1,\ldots,n-1.
\end{equation}
%


 \setcounter{equation}{0}
\section{Spectral properties of the matrices in ${\cal J}_n$}\label{section:spectrum.odd.even}

In this section, we characterize the spectra of the matrices in~${\cal J}_n$. Note first that any matrix $J_n$
of the form~\eqref{main.matrix.2} can be represented as follows
\begin{equation*}\label{main.submatrix.2}
J_n=
\begin{pmatrix}
J_{n-1} & \mathbf{y}\\
\overline{\mathbf{y}}^{T} & a_n\\
\end{pmatrix},
\end{equation*}
where $\mathbf{y}=(\overline{b}_n,0,0,\ldots,0,b_{n-1})^{T}\in\mathbb{C}^{n-1}$.
By the \emph{Schur determinant formula}~\cite{Zhang}, we obtain~(cf.~\cite{XJO})
\begin{equation}\label{main.rat.func.0}
\dfrac{\chi_n(\lambda)}{\chi_{n-1}(\lambda)}=\lambda-a_n-
\overline{\mathbf{y}}^T(\lambda I_{n-1}-J_{n-1})^{-1}\mathbf{y},
\end{equation}
where $\chi_n(\lambda)$ and $\chi_{n-1}(\lambda)$ are the characteristic polynomials
of the matrices $J_{n}$ and $J_{n-1}$, respectively:
\begin{equation*}
\chi_n(\lambda)=\det(\lambda I_{n}-J_{n}),\quad \chi_{n-1}(\lambda)=\det(\lambda I_{n-1}-J_{n-1}).
\end{equation*}
Furthermore, the formula~\eqref{submatrix.resolvent} implies
\begin{equation}\label{submatrix.resolvent.0}
\overline{\mathbf{y}}^T(\lambda
I_{n-1}-J_{n-1})^{-1}\mathbf{y}=\sum\limits_{k=1}^{n-1}\dfrac{|\overline{\mathbf{y}}^{T}\mathbf{u}_k|^2}{\lambda-\mu_k}=
\sum\limits_{k=1}^{n-1}\dfrac{|b_nu_{k1}+\overline{b}_{n-1}u_{k,n-1}|^2}{\lambda-\mu_k}.
\end{equation}
From~\eqref{main.rat.func.0} and~\eqref{submatrix.resolvent.0} we get
\begin{equation}\label{main.rat.func}
\dfrac{\chi_n(\lambda)}{\chi_{n-1}(\lambda)}=\lambda-a_n-
\sum\limits_{k=1}^{n-1}\dfrac{\alpha_k}{\lambda-\mu_k},
\end{equation}
where
\begin{equation}\label{residues.0}
\alpha_k=-\dfrac{\chi_{n}(\mu_k)}{\chi'_{n-1}(\mu_k)}=|b_nu_{k1}+\overline{b}_{n-1}u_{k,n-1}|^2\geqslant0
\end{equation}
or
\begin{equation}\label{residues}
\alpha_k=|b_nu_{k1}|^2+|b_{n-1}u_{k,n-1}|^2+2\Re(b_nu_{k1}b_{n-1}\overline{u}_{k,n-1})\geqslant0.
\end{equation}
Note that~\eqref{residues.0}--\eqref{residues} and~\eqref{main.submatrix.2.spectrum.signs} imply that
\begin{equation}\label{signs.chi_n}
(-1)^{n-k}\chi_{n}(\mu_k)\geqslant0.
\end{equation}

Thus, the function
\begin{equation*}
R_{n-1}(\lambda)=\lambda-\sum\limits_{k=1}^{n-1}\dfrac{\alpha_k}{\lambda-\mu_k}
\end{equation*}
maps the upper half-plane of the complex plane to itself, since
\begin{equation*}
\sgn\Im\dfrac{-\alpha_k}{\lambda-\mu_k}=\sgn\Im\dfrac{-\alpha_k(\overline{\lambda}-\mu_k)}{|\lambda-\mu_k|^2}=\sgn\Im\lambda.
\end{equation*}
Therefore, the zeroes and poles of $R_{n-1}(\lambda)$ are real,
simple and interlacing (see e.g.~\cite{Tyaglov} and references
therein), so~\eqref{main.rat.func} can be represented as follows
\begin{equation}\label{main.rat.func.222}
\dfrac{\chi_n(\lambda)}{\chi_{n-1}(\lambda)}=R_{n-1}(\lambda)-a_n.
\end{equation}
By~\eqref{residues.0}, $\chi_{n}(\mu_j)=0$ for some $j$, $j=1,\ldots,n-1$, if and only if $\alpha_j=0$, or
equivalently, if and only if
\begin{equation*}
b_nu_{k1}+\overline{b}_{n-1}u_{k,n-1}=0.
\end{equation*}

Let $N=\{j_1,\ldots,j_m\}\subset\{1,2,\ldots,n-1\}$ be the set of
indices such that $\alpha_j=0$ for $j\in N$. Then the function
$R_{n-1}(\lambda)$ has the form
\begin{equation}\label{main.rat.func.general.case}
R_{n-1}(\lambda)=\lambda-\sum\limits_{\substack{k=1\\k\not\in N}}^{n-1}\dfrac{\alpha_k}{\lambda-\mu_k},
\end{equation}
and the polynomial $\chi_n(\lambda)$ has $m$ zeroes in common with
$\chi_{n-1}(\lambda)$, $\mu_{j_1}$, $\mu_{j_2}$, \ldots,
$\mu_{j_m}$ while the its other zeroes are the
solutions of the equation
\begin{equation}\label{main.equation}
R_{n-1}(\lambda)=a_n.
\end{equation}

If $a_n\in\mathbb{R}$, then the function $R_{n-1}(\lambda)-a_n$ maps the upper half-plane to itself,
so the solutions of the equation~\eqref{main.equation} are real and simple and interlace the
numbers~$\mu_k$, $k\in\{1,2,\ldots,n-1\}\backslash N$. However, they may coincide with some
numbers $\mu_j$, $j\in N$. So for $a_n\in\mathbb{R}$, the eigenvalues of $J_n$ are real and
of multiplicity at most two. Multiple eigenvalues are always eigenvalues of $J_{n-1}$.
This property of the spectrum of $J_n$ with $b_k\in\mathbb{C}$, $k=1,2\ldots,n$ is the same as
for real $b_k$, e.g.~\cite{XJO}.

Suppose now that $a_n$ is nonreal, and recall that
$R_{n-1}(\lambda)$ maps the upper (lower) half-plane to itself.
Therefore, all solutions of the equation~\eqref{main.equation} lie
in the upper (lower) half-plane of the complex plane whenever
$\Im{a_n}>0$ ($\Im{a_n}<0$), and cannot be real since the functions
mapping the upper half-plane to itself are always real on the real
line (see e.g.~\cite{Tyaglov} and references there). Note that the
famous Hermite-Biehler theorem can be proved by the same technique,
see~\cite{Obreschkoff}.

Finally, we note that the formul\ae~\eqref{residues.0} and~\eqref{formula.vectors} imply that
\begin{equation}\label{residues.1}
\begin{array}{c}
\alpha_k=\dfrac{[|b_1u_{k1}|^2+\Re(b_nu_{k1}b_{n-1}\overline{u}_{k,n-1})]^2+[\Im(b_nu_{k1}b_{n-1}\overline{u}_{k,n-1})]^2}{|b_1u_{k1}|^2}=\\
\\
=\dfrac{[\chi'_{n-1}(\mu_k)|b_1u_{k1}|^2+\Re\beta]^2+(\Im\beta)^2}{|b_1u_{k1}|^2\left[\chi'_{n-1}(\mu_k)\right]^2}.
\end{array}
\end{equation}
This formula shows that if $\Im\beta\neq0$, then $\alpha_k\neq0$
for any $k=1,\ldots,n-1$. At the same time, if $\beta\in\mathbb{R}\backslash\{0\}$, then
\begin{equation*}\label{residues.101}
\alpha_k=\dfrac{[\chi'_{n-1}(\mu_k)|b_1u_{k1}|^2+\beta]^2}{|b_1u_{k1}|^2\left[\chi'_{n-1}(\mu_k)\right]^2}.
\end{equation*}
Therefore, for real nonzero $\beta$, the number $\alpha_k$ equals zero if and only if
$\chi'_{n-1}(\mu_k)|b_1u_{k1}|^2+\beta=0$, that is equivalent to the following equality
\begin{equation*}\label{residues.1011}
|\chi'_{n-1}(\mu_k)||b_1u_{k1}|^2=(-1)^{n-k}\beta,
\end{equation*}
by~\eqref{main.submatrix.2.spectrum.signs}. This equality
and~\eqref{signs.chi_n} give us:

\noindent for $\beta>0$
\begin{equation}\label{signs.chi_n.beta.pos}
\begin{cases}
&(-1)^{n-k}\chi_n(\mu_k)>0,\qquad k=n-1,n-3,n-5,\ldots,\\
&(-1)^{n-k}\chi_n(\mu_k)\geqslant0,\qquad k=n-2,n-4,n-6,\ldots,\\
\end{cases}
\end{equation}

\noindent and for $\beta<0$
\begin{equation}\label{signs.chi_n.beta.neg}
\begin{cases}
&(-1)^{n-k}\chi_n(\mu_k)\geqslant0,\qquad k=n-1,n-3,n-5,\ldots,\\
&(-1)^{n-k}\chi_n(\mu_k)>0,\qquad k=n-2,n-4,n-6,\ldots.\\
\end{cases}
\end{equation}

If additionally $a_n\in\mathbb{R}$, then the eigenvalues of the
matrices $J_n$ and $J_{n-1}$ are distributed as follows:

\vspace{3mm}

\noindent for $\beta<0$
\begin{equation}\label{residues.10111}
\cdots\leqslant\lambda_{n-5}<\mu_{n-4}<\lambda_{n-4}\leqslant\mu_{n-3}\leqslant\lambda_{n-2}
<\mu_{n-2}<\lambda_{n-1}\leqslant\mu_{n-1}\leqslant\lambda_n,
\end{equation}

\noindent and for $\beta>0$
\begin{equation}\label{residues.10112}
\cdots<\lambda_{n-5}\leqslant\mu_{n-4}\leqslant\lambda_{n-4}<\mu_{n-3}<\lambda_{n-2}
\leqslant\mu_{n-2}\leqslant\lambda_{n-1}<\mu_{n-1}<\lambda_n.
\end{equation}

Thus, we come to the following statements.

\begin{theorem}
Let $\beta$ be nonreal. If $a_n\in\mathbb{R}$, then all the
eigenvalues of the matrix $J_n$ defined in~\eqref{main.matrix.2} are
real and simple and interlace the eigenvalues of~$J_{n-1}$, which
are real and simple as well.

If $a_n\in\mathbb{C}\backslash\mathbb{R}$, then all the eigenvalues
of the matrix $J_n$ lie in the \textbf{open} upper (lower)
half-plane of the complex plane whenever $\Im{a_n}>0$
$(\Im{a_n}<0)$.

Moreover, for complex $a_n$ (real or nonreal) the characteristic polynomial of the matrix $J_n$
satisfies the inequalities
\begin{equation*}
(-1)^{n-k}\chi_n(\mu_k)>0,\qquad k=1,\ldots,n-1.
\end{equation*}
\end{theorem}
\begin{theorem}
Let $\beta\in\mathbb{R}$. If $a_n\in\mathbb{R}$, then all the
eigenvalues of the matrix $J_n$ given in~\eqref{main.matrix.2} are
real and of multiplicity at most $2$. Any multiple eigenvalue of
$J_n$ is an eigenvalue of~$J_{n-1}$.

If $a_n\in\mathbb{C}\backslash\mathbb{R}$, then all the eigenvalues
of the matrix $J_n$ lie in the \textbf{closed} upper (lower)
half-plane of the complex plane whenever $\Im{a_n}>0$
$(\Im{a_n}<0)$. An eigenvalue of~$J_n$ is real if and only if it is
an eigenvalue of $J_{n-1}$.

Moreover, the eigenvalues of $J_n$ and $J_{n-1}$ are distributed as in~\eqref{residues.10111}--\eqref{residues.10112}
for real $a_n$, or satisfy the inequalities~\eqref{signs.chi_n.beta.pos}--\eqref{signs.chi_n.beta.neg} for any complex $a_n$.
\end{theorem}

\begin{remark}\label{remark.imaginary.part.charpoly}
From~\eqref{main.rat.func} it follows that if $a_n$ is nonreal, then
\begin{equation}\label{char.poly.full.formula}
\chi_n(\lambda)=\chi_{n-1}(\lambda)(\lambda-\Re a_n)-
\sum\limits_{k=1}^{n}\alpha_k\prod\limits_{\substack{j=1\\j\neq k}}^{n-1}(\lambda-\mu_j)-i\Im a_n\cdot\chi_{n-1}(\lambda),
\end{equation}
that is, $\chi_{n-1}(\lambda)$ is the imaginary part (up to the constant factor $\Im a_n$) of the polynomial~$\chi_n(\lambda)$,
so $\chi_n(\mu_k)\in\mathbb{R}$, $k=1,\ldots,n-1$.

%
%
\end{remark}

\vspace{3mm}

We now study a necessary condition which the spectra of the matrices
in the class~${\cal J}_n$ must satisfy.

\begin{theorem}\label{Theorem.spectrum.necessary.condition}
Let $J_n\in{\cal J}_n$, $\chi_n(\lambda)=\det(\lambda I_n-J_n)$, and
$\sigma(J_{n-1})=\{\mu_1,\ldots,\mu_{n-1}\}$. Then
\begin{itemize}
%
%
%
\item for nonreal $\beta$
\begin{equation}\label{spectrum.necessary.condition.nonreal.2}
(-1)^{n-k}\chi_n(\mu_k)>0\quad\text{and}\quad|\chi_n(\mu_k)+2\Re\beta|\geqslant2|\beta|,\quad k=1,\ldots,n-1,
\end{equation}
%
%
%
\item for real nonzero $\beta$
\begin{equation}\label{spectrum.necessary.condition.real.2}
(-1)^{n-k}\chi_n(\mu_k)\geqslant0\quad\text{and}\quad|\chi_n(\mu_k)|\geqslant4(-1)^{n-k-1}\beta,\quad k=1,\ldots,n-1,
\end{equation}
\end{itemize}

Here $\beta$ is defined in~\eqref{beta}:
\begin{equation*}
\beta=b_1b_2\cdots b_n\neq0.
\end{equation*}
\end{theorem}
\begin{proof}
In fact, taking into account that $b_nu_{k1}\neq0$ by~\eqref{formula.vectors} and
$\chi'_{n-1}(\mu_k)\neq0$ by~\eqref{main.submatrix.2.spectrum.signs}, from~\eqref{residues.0}
and~\eqref{residues.1} we obtain that $X_k:=|b_nu_{k1}|^2$ satisfies
the following equation
\begin{equation}\label{Theorem.spectrum.necessary.condition.proof.equation}
[\chi'_{n-1}(\mu_k)X_k+\Re\beta]^2+\chi_n(\mu_k)\chi'_{n-1}(\mu_k)X_k+(\Im\beta)^2=0,
\end{equation}
or
\begin{equation}\label{Theorem.spectrum.necessary.condition.proof.equation1}
[\chi'_{n-1}(\mu_k)]^2X^2_k+(\chi_n(\mu_k)+2\Re\beta)\chi'_{n-1}(\mu_k)X_k+(\Im\beta)^2=0.
\end{equation}

The solutions of this equations have the form
\begin{equation}\label{solutions.new}
X_k^{(1,2)}=\dfrac{(-1)^{n-k}(\chi_{n}(\mu_k)+2\Re\beta)
\pm\sqrt{[\chi_{n}(\mu_k)+2\Re\beta]^2-4|\beta|^2}}{2|\chi'_{n-1}(\mu_k)|}.
\end{equation}

Since $b_nu_{k1}\neq0$, $X_k$ is positive, so we have for any (real or nonreal) $\beta$ and $a_n$ that (see~\eqref{signs.chi_n})
\begin{equation}\label{Theorem.spectrum.necessary.condition1.proof.1}
[\chi_{n}(\mu_k)+2\Re\beta]^2-4|\beta|^2\geqslant0,
\end{equation}
\begin{equation}\label{Theorem.spectrum.necessary.condition2.proof.1}
(-1)^{n-k}(\chi_{n}(\mu_k)+2\Re\beta)\geqslant0,
\end{equation}
\begin{equation}\label{Theorem.spectrum.necessary.condition3.proof.1}
(-1)^{n-k}\chi_{n}(\mu_k)\geqslant0,
\end{equation}
and $[\chi_{n}(\mu_k)+2\Re\beta]^2-4|\beta|^2\geqslant0$ and
$(-1)^{n-k}(\chi_{n}(\mu_k)+2\Re\beta)\geqslant0$ do not equal
zero simultaneously.

We now show that the inequality in~\eqref{Theorem.spectrum.necessary.condition2.proof.1}
follows from~\eqref{Theorem.spectrum.necessary.condition1.proof.1}
and~\eqref{Theorem.spectrum.necessary.condition3.proof.1}. Indeed, the
inequality~\eqref{Theorem.spectrum.necessary.condition1.proof.1} implies
\begin{equation*}
(\chi_{n}(\mu_k)+4\Re\beta)\chi_{n}(\mu_k)\geqslant4(\Im\beta)^2\geqslant0,
\end{equation*}
so
\begin{equation*}
(-1)^{n-k}(\chi_{n}(\mu_k)+4\Re\beta)\geqslant0.
\end{equation*}
Therefore, if $(-1)^{n-k}\Re\beta<0$, then
\begin{equation*}
(-1)^{n-k}(\chi_{n}(\mu_k)+2\Re\beta)>(-1)^{n-k}(\chi_{n}(\mu_k)+4\Re\beta)\geqslant0,
\end{equation*}
and if $(-1)^{n-k}\Re\beta\geqslant0$, then
\begin{equation*}
(-1)^{n-k}(\chi_{n}(\mu_k)+2\Re\beta)\geqslant(-1)^{n-k}\chi_{n}(\mu_k)\geqslant0.
\end{equation*}

Thus, we obtain that for any complex (real or nonreal) $\beta$ and $a_n$,
the eigenvalues of the matrices $J_n$ and $J_{n-1}$ satisfy the
inequalities~\eqref{Theorem.spectrum.necessary.condition1.proof.1}
and~\eqref{Theorem.spectrum.necessary.condition3.proof.1} that is
equivalent to~\eqref{spectrum.necessary.condition.nonreal.2}.

Let now $\beta\in\mathbb{R}\backslash\{0\}$, so $\Re\beta=\beta$ and $\Im\beta=0$. In this case, the
inequality~\eqref{Theorem.spectrum.necessary.condition1.proof.1}
has the form
\begin{equation*}
\chi_{n}(\mu_k)[\chi_{n}(\mu_k)+4\beta]\geqslant0,
\end{equation*}
so
\begin{equation}\label{Theorem.spectrum.necessary.condition.proof.3}
\chi_{n}(\mu_k)=0,
\end{equation}
or
\begin{equation}\label{Theorem.spectrum.necessary.condition.proof.4}
|\chi_{n}(\mu_k)|\geqslant4\beta(-1)^{n-k-1}.
\end{equation}
Moreover, for real $\beta$ the eigenvalues of $J_n$ and $J_{n-1}$
satisfy the
inequalities~\eqref{signs.chi_n.beta.pos}--\eqref{signs.chi_n.beta.neg},
which include~\eqref{Theorem.spectrum.necessary.condition.proof.3}.

It is easy to see
that~\eqref{signs.chi_n.beta.pos}--\eqref{signs.chi_n.beta.neg}
and~\eqref{Theorem.spectrum.necessary.condition.proof.4}
imply~\eqref{spectrum.necessary.condition.real.2}, as required.
\end{proof}
\begin{remark}
By the Hermite--Biehler theorem (see e.g.~\cite{Obreschkoff}), the condition $(-1)^{n-k}\chi_{n}(\mu_k)>0$
in~\eqref{spectrum.necessary.condition.nonreal.2}
can be changed to  $\chi_{n}(\mu_k)\in\mathbb{R}$,
since for $\Im a_n>0$ ($\Im a_n<0$) the polynomial $\chi_n(\lambda)$
has all roots in the open upper (lower) half-plane of the complex plane,
and since $\chi_{n-1}(\lambda)$ is its imaginary part (see Remark~\ref{remark.imaginary.part.charpoly}).
\end{remark}
\begin{remark}
The inequalities~\eqref{spectrum.necessary.condition.real.2} were established
in~\cite{XJO}  under the assumption of reality of all the entries of
the matrix $J_n$ and $b_n>0$.
\end{remark}
\begin{remark}
If $\beta$ is a pure imaginary number, that is, $\beta=i\gamma$,
$\gamma\in\mathbb{R}\backslash\{0\}$, then the
conditions~\eqref{spectrum.necessary.condition.nonreal.2} have the
form
\begin{equation*}\label{spectrum.necessary.condition.imaginary}
|\chi_n(\mu_k)|\geqslant2|\gamma|,\quad k=1,\ldots,n-1.
\end{equation*}
\end{remark}

\begin{remark}\label{Remark.solutions}
The formul\ae~\eqref{residues.0} and~\eqref{formula.vectors} also imply that
\begin{equation}\label{Remark.solutions.eq}
\begin{array}{c}
\alpha_k=\dfrac{[|\overline{b}_{n-1}u_{k,n-1}|^2+\Re(b_nu_{k1}b_{n-1}\overline{u}_{k,n-1})]^2+
[\Im(b_nu_{k1}b_{n-1}\overline{u}_{k,n-1})]^2}{|\overline{b}_{n-1}u_{k,n-1}|^2}=\\
\\
=\dfrac{[\chi'_{n-1}(\mu_k)|\overline{b}_{n-1}u_{k,n-1}|^2+\Re\beta]^2+
(\Im\beta)^2}{|\overline{b}_{n-1}u_{k,n-1}|^2\left[\chi'_{n-1}(\mu_k)\right]^2}.
\end{array}
\end{equation}
Hence from~\eqref{residues.0} and~\eqref{Remark.solutions.eq} we get
that $X_k:=|b_{n-1}u_{k,n-1}|^2$ satisfies the
equation~\eqref{Theorem.spectrum.necessary.condition.proof.equation}.
Therefore, if $X^{(1,2)}_k$ are solutions
of~\eqref{Theorem.spectrum.necessary.condition.proof.equation}, then
$X^{(1)}_k=|b_1u_{k1}|^2$,
$X^{(2)}_k=|\overline{b}_{n-1}u_{k,n-1}|^2$, or
$X^{(1)}_k=|\overline{b}_{n-1}u_{k,n-1}|^2$,
$X^{(2)}_k=|b_1u_{k1}|^2$.
\end{remark}

\vspace{4mm}

Finally, recall that the system of vectors
$\mathbf{u}_1,\ldots,\textbf{u}_{n-1}$  is orthonormal, so the
matrix $U=[|u_{kj}|]_{k,j=1}^{n-1}$ is unitary. Therefore,
\begin{equation*}
\sum\limits_{k=1}^{n-1}|u_{k1}|^2=1.
\end{equation*}
Since  $X_k=|b_nu_{k1}|^2$, we have
\begin{equation}\label{Theorem.spectrum.necessary.condition.proof.8}
\sum\limits_{k=1}^{n-1}X_k^{(1,2)}=\sum\limits_{k=1}^{n-1}|b_nu_{k1}|^2=|b_n|^2\sum\limits_{k=1}^{n-1}|u_{k1}|^2=|b_n|^2.
\end{equation}
Thus, \eqref{solutions.new} and~\eqref{Theorem.spectrum.necessary.condition.proof.8} imply
\begin{equation}\label{abs.b_n.formula}
|b_n|=\left(\sum\limits_{k=1}^{n-1}\dfrac{(-1)^{n-k}(\chi_{n}(\mu_k)+2\Re\beta)
\pm\sqrt{[\chi_{n}(\mu_k)+2\Re\beta]^2-4|\beta|^2}}{2|\chi'_{n-1}(\mu_k)|}\right)^{\tfrac12},
\end{equation}
and
\begin{equation}\label{u_k1.formula}
|u_{k1}|^2=\dfrac{(-1)^{n-k}(\chi_{n}(\mu_k)+2\Re\beta)
\pm\sqrt{[\chi_{n}(\mu_k)+2\Re\beta]^2-4|\beta|^2}}{2|b_n|^2|\chi'_{n-1}(\mu_k)|}>0,
\end{equation}
for $k=1,\ldots,n-1$.

By the same reasoning as above and by Remark~\ref{Remark.solutions} we obtain that
\begin{equation}\label{abs.b_n-1.formula}
|b_{n-1}|=\left(\sum\limits_{k=1}^{n-1}\dfrac{(-1)^{n-k}(\chi_{n}(\mu_k)+2\Re\beta)
\mp\sqrt{[\chi_{n}(\mu_k)+2\Re\beta]^2-4|\beta|^2}}{2|\chi'_{n-1}(\mu_k)|}\right)^{\tfrac12},
\end{equation}
and
\begin{equation}\label{u_k,n-1.formula}
|u_{k,n-1}|^2=\dfrac{(-1)^{n-k}(\chi_{n}(\mu_k)+2\Re\beta)
\mp\sqrt{[\chi_{n}(\mu_k)+2\Re\beta]^2-4|\beta|^2}}{2|b_{n-1}|^2|\chi'_{n-1}(\mu_k)|}>0,
\end{equation}
for $k=1,\ldots,n-1$.

If $\beta\in\mathbb{R}\backslash\{0\}$, then the
formul\ae~\eqref{abs.b_n.formula}--
\eqref{u_k1.formula}
can be represented in the following form:
\begin{equation}\label{abs.b_n.formula.beta.real}
|b_n|=\left(\sum\limits_{k=1}^{n-1}\dfrac{(-1)^{n-k}(\chi_{n}(\mu_k)+2\beta)
\pm\sqrt{\chi_{n}(\mu_k)(\chi_{n}(\mu_k)+4\beta)}}{2|\chi'_{n-1}(\mu_k)|}\right)^{\tfrac12},
\end{equation}
and
\begin{equation}\label{u_k1.formula.beta.real}
|u_{k1}|^2=\dfrac{(-1)^{n-k}(\chi_{n}(\mu_k)+2\beta)
\pm\sqrt{\chi_{n}(\mu_k)(\chi_{n}(\mu_k)+4\beta)}}{2|b_n|^2|\chi'_{n-1}(\mu_k)|}>0,
\end{equation}
for $k=1,\ldots,n-1$.
\begin{equation}\label{abs.b_n-1.formula.beta.real}
|b_{n-1}|=\left(\sum\limits_{k=1}^{n-1}\dfrac{(-1)^{n-k}(\chi_{n}(\mu_k)+2\beta)
\mp\sqrt{\chi_{n}(\mu_k)(\chi_{n}(\mu_k)+4\beta)}}{2|\chi'_{n-1}(\mu_k)|}\right)^{\tfrac12},
\end{equation}
and
\begin{equation}\label{u_k,n-1.formula.beta.real}
|u_{k,n-1}|^2=\dfrac{(-1)^{n-k}(\chi_{n}(\mu_k)+2\beta)
\mp\sqrt{\chi_{n}(\mu_k)(\chi_{n}(\mu_k)+4\beta)}}{2|b_{n-1}|^2|\chi'_{n-1}(\mu_k)|}>0,
\end{equation}
for $k=1,\ldots,n-1$.

\setcounter{equation}{0}

\section{Inverse problems for matrices in  ${\cal J}_n$ }\label{section:inverse.problems.even.odd}%

In this section, we show that the necessary conditions~\eqref{spectrum.necessary.condition.nonreal.2}--\eqref{spectrum.necessary.condition.real.2} on the spectra of the matrices in the class~${\cal J}_n$ are also sufficient. This means that
given $2n$ numbers $\lambda_1$, \ldots, $\lambda_n$, $\mu_1$, \ldots, $\mu_{n-1}$, $\beta$
satisfying~\eqref{spectrum.necessary.condition.nonreal.2} or~\eqref{spectrum.necessary.condition.real.2}
we can reconstruct a matrix $J_n\in{\cal J}_n$ whose eigenvalues are $\lambda_1$, \ldots, $\lambda_n$,
the eigenvalues of its leading principal submatrix $J_{n-1}$ are $\mu_1$, \ldots, $\mu_{n-1}$, and
$\beta=b_1\cdots b_n$.

But any matrix in the class ${\cal J}_n$ has $3n+1$ parameters,
namely, $c_k$, $k=1,\ldots,n-1$, $\Re a_n$, $\Im a_n$, $\Re b_j$,
$\Im b_j$, $j=1,\ldots,n$. At the same time, the spectral data
$\lambda_1$, \ldots, $\lambda_n$, $\mu_1$, \ldots, $\mu_{n-1}$,
$\beta$ consist of only $2n$ or $2n+1$ parameters. This is obvious
if all the $\lambda_j$ are real. But if the $\lambda_j$ are nonreal,
they depend on the numbers $\mu_k$, since the numbers
$\prod\limits_{j=1}^{n}(\mu_k-\lambda_j)$ are real
by~\eqref{spectrum.necessary.condition.nonreal.2}. Thus, if we want
to reconstruct finitely many matrices from the spectral data mentioned above, we
should restrict ourselves to a subclass of the class~${\cal J}_n$.

Thus, in this section, we consider the subclass~$\widehat{{\cal J}}_n$ of the class~${\cal J}_n$
consisting of matrices of the form
\begin{equation}\label{main.matrix.subclass}
\widehat{J}_n=\begin{pmatrix}
\widehat{c}_1 & \widehat{b}_1 &  0 &\dots&   0   & \overline{\widehat{b}}_n \\
\widehat{b}_1 & \widehat{c}_2 &\widehat{b}_2 &\dots&   0   & 0 \\
0  &\widehat{b}_2 & \widehat{c}_3 &\dots&   0   & 0 \\
\vdots&\vdots&\vdots&\ddots&\vdots&\vdots\\
0  &  0  &  0  &\dots&\widehat{c}_{n-1}& \widehat{b}_{n-1}\\
\widehat{b}_n  &  0  &  0  &\dots&\widehat{b}_{n-1}&\widehat{a}_n\\
\end{pmatrix},
\end{equation}
with \emph{real} $\hat c_k,\hat b_k$, $k=1,\ldots,n-1$ and complex
(real or nonreal) $\hat a_n$ and $\hat b_n$. This subclass has  the
following important property.

\begin{lemma}
For any matrix in the class ${\cal J}_n$, there exists
a matrix in the class $\widehat{{\cal J}}_n$ with the same spectral data, that is, with
the same spectra of the matrix itself and of its leading $(n-1)\times(n-1)$ submatrix,
and the same number~$\beta$ $($see~\eqref{beta}$)$.
\end{lemma}
\begin{proof}
Indeed, by~\eqref{char.poly.full.formula} the spectrum of a matrix $J_n$ in the class ${\cal J}_n$ depends
on $a_n$, $\mu_k$, and $\alpha_k$, $k=1,\ldots,n-1$, where $\mu_k$ are the eigenvalues of the leading principal
submatrix $J_{n-1}$. At the same time, by~\eqref{residues.1}
the numbers~$\alpha_k$ depend on $\mu_k$, $\beta$, and~$|b_1u_{k1}|^2$, where $\beta$ is defined in~\eqref{beta}.

Thus, given a matrix $J_n$ in the class ${\cal J}_n$, we must find a
matrix $\widehat{J}_n$ in the class $\widehat{{\cal J}}_n$ with the same $a_n$, $\beta$,
the spectrum of the submatrix $\widehat{J}_{n-1}$, and the same numbers $|b_1u_{k1}|^2$.

So we consider a matrix
\begin{equation*}
J_n=\begin{pmatrix}
c_1 & b_1 &  0 &\dots&   0   & \overline{b}_n \\
\overline{b}_1 & c_2 &b_2 &\dots&   0   & 0 \\
0  &\overline{b}_2 & c_3 &\dots&   0   & 0 \\
\vdots&\vdots&\vdots&\ddots&\vdots&\vdots\\
0  &  0  &  0  &\dots&c_{n-1}& b_{n-1}\\
b_n  &  0  &  0  &\dots&\overline{b}_{n-1}&a_n\\
\end{pmatrix},
\end{equation*}
in the class ${\cal J}_n$ and construct the matrix
\begin{equation*}
\widehat{J}_n=\begin{pmatrix}
\widehat{c}_1 & \widehat{b}_1 &  0 &\dots&   0   & \overline{\widehat{b}}_n \\
\widehat{b}_1 & \widehat{c}_2 &\widehat{b}_2 &\dots&   0   & 0 \\
0  &\widehat{b}_2 & \widehat{c}_3 &\dots&   0   & 0 \\
\vdots&\vdots&\vdots&\ddots&\vdots&\vdots\\
0  &  0  &  0  &\dots&\widehat{c}_{n-1}& \widehat{b}_{n-1}\\
\widehat{b}_n  &  0  &  0  &\dots&\widehat{b}_{n-1}&\widehat{a}_n\\
\end{pmatrix},
\end{equation*}
such that
\begin{equation}\label{new.matrix.entries.1}
\begin{array}{c}
\widehat{c}_j=c_j,\qquad j=1,\ldots,n-1,\\
\\
\widehat{b}_k=|b_k|,\qquad k=1,\ldots,n-1,\\
\\
\widehat{a}_n=a_n,
\end{array}
\end{equation}
and
\begin{equation}\label{new.matrix.entries.2}
\widehat{b}_n=\dfrac{b_1b_2\cdots b_n}{|b_1b_2\cdots b_{n-1}|}=\dfrac{\beta}{\widehat{b}_1\widehat{b}_2\cdots \widehat{b}_{n-1}}.
\end{equation}
Obviously, $\widehat{J}_n\in\widehat{{\cal J}}_n$.

The formul\ae~\eqref{new.matrix.entries.1} show that the spectra of the submatrices $J_{n-1}$
and $\widehat{J}_{n-1}$ coincide (see~\cite{KreinGantmacher}):
\begin{equation*}
\sigma(\widehat{J}_n)=\sigma(J_n)=\{\mu_1,\ldots,\mu_{n-1}\}.
\end{equation*}
Moreover, from~\eqref{new.matrix.entries.2} we have
\begin{equation*}
\beta=b_1b_2\cdots b_n=\widehat{b}_1\widehat{b}_2\cdots \widehat{b}_n=\widehat{\beta}.
\end{equation*}
Therefore, to establish that $\sigma(\widehat{J}_n)=\sigma(J_n)$ it suffices to prove that
\begin{equation}\label{residues.equality}
\left(\widehat{b}_1\widehat{u}_{k1}\right)^2=|b_1u_{k1}|^2,\qquad k=1,\ldots,n-1.
\end{equation}

The formula~\eqref{submatrix.resolvent} and the equalities~\eqref{new.matrix.entries.1} imply that
\begin{equation*}
\begin{array}{c}
\sum\limits_{k=1}^{n-1}\dfrac{|b_1u_{k1}|^2}{\lambda-\mu_k}=|b_1|^2\mathbf{e}_1^{T}(\lambda I_{n-1}-J_{n-1})^{-1}\mathbf{e}_1=\\
\\
=\dfrac{|b_1|^2\det(\lambda I_{n-2}-J_{n-2})}{\det(\lambda I_{n-1}-J_{n-1})}=
\dfrac{\widehat{b}_1^2\det(\lambda I_{n-2}-\widehat{J}_{n-2})}{\det(\lambda I_{n-1}-\widehat{J}_{n-1})}=
\\
=\widehat{b}_1^2\mathbf{e}_1^{T}(\lambda I_{n-1}-\widehat{J}_{n-1})^{-1}\mathbf{e}_1=
\sum\limits_{k=1}^{n-1}\dfrac{\left(\widehat{b}_1\widehat{u}_{k1}\right)^2}{\lambda-\mu_k},
\end{array}
\end{equation*}
so~\eqref{residues.equality} holds, as required.
\end{proof}

Now we establish that the conditions~\eqref{spectrum.necessary.condition.nonreal.2}--\eqref{spectrum.necessary.condition.real.2}
on the spectral data are also sufficient to reconstruct finitely
many matrices in the class~$\widehat{{\cal J}}_n$ with this spectral data.
We consider 4 cases.
\begin{theorem}\label{Theorem.inverse.beta.a_n.nonreal}
Let $\beta$ be a nonreal number, and let $\{\mu_k\}_{k=1}^{n-1}$ be a set of
real distinct numbers.
%
%
Suppose that $\{\lambda_j\}_{j=1}^{n}$  is a set of complex
numbers such that $\Im\lambda_j>0$, $j=1,\ldots,n-1$.

Then there exists a matrix $\widehat{J}_n$ in the class $\widehat{{\cal J}}_n$
with $\Im \widehat{a}_n>0$ and $\beta=\widehat{b}_1\cdots \widehat{b}_n$ such that
$\sigma(\widehat{J}_n)=\{\lambda_1,\ldots,\lambda_n\}$
and $\sigma(\widehat{J}_{n-1})=\{\mu_1,\ldots,\mu_{n-1}\}$ if and only if
\begin{equation}\label{solvability.condition.beta.a_n.nonreal}
(-1)^{n-k}\chi_n(\mu_k)>0\quad\text{and}\quad|\chi_n(\mu_k)+2\Re\beta|\geqslant2|\beta|,\quad k=1,\ldots,n-1,
\end{equation}
where $\chi_n(\lambda)=\prod\limits_{j=1}^{n}(\lambda-\lambda_j)$.


Moreover, if there are $m$ $(0\leqslant m\leqslant n-1)$ equalities
in~\eqref{solvability.condition.beta.a_n.nonreal}, then there exist
exactly $2^{n-m-1}$ matrices in the class $\widehat{{\cal J}}_n$ with this spectral data.
\end{theorem}
\begin{proof}
The necessity of the conditions~\eqref{solvability.condition.beta.a_n.nonreal} is
provided by Theorem~\ref{Theorem.spectrum.necessary.condition}.

Suppose that we have a nonreal number $\beta$, distinct real numbers $\{\mu_k\}_{k=1}^{n-1}$,
and the numbers $\{\lambda_j\}_{j=1}^{n}$ in the open upper half-plane, and suppose that
they satisfy~\eqref{solvability.condition.beta.a_n.nonreal}.

We also assume that there are exactly $m$, $0\leqslant m\leqslant
n-1$, equalities in~\eqref{solvability.condition.beta.a_n.nonreal}.
Then one has exactly $2^{n-m-1}$ ways to construct the number
$|\widehat{b}_n|$ and the corresponding values $|u_{k1}|^2$,
$k=1,\ldots,n-1$, by the formul\ae~\eqref{abs.b_n.formula}
and~\eqref{u_k1.formula}. The
conditions~\eqref{solvability.condition.beta.a_n.nonreal} ensure
positivity of the obtained numbers $|\widehat{b}_n|$ and
$|u_{k1}|^2$, $k=1,\ldots,n-1$. Indeed,
\eqref{solvability.condition.beta.a_n.nonreal} imply positivity of
the numbers $(-1)^{n-k}[\chi_{n-1}(\mu_k)+2\Re\beta]$,
$k=1,\ldots,n-1$, as we established in the proof of
Theorem~\ref{Theorem.spectrum.necessary.condition}. At the same
time, positivity of these numbers and the
conditions~\eqref{solvability.condition.beta.a_n.nonreal} obviously
imply positivity of the numbers $X_k^{(1,2)}$ defined
in~\eqref{solutions.new}. Thus, the
conditions~\eqref{solvability.condition.beta.a_n.nonreal} guaranty
the existence of positive numbers $|\widehat{b}_n|$ and
$|u_{k1}|^2$, $k=1,\ldots,n-1$.

With the values $|u_{k1}|^2$, $k=1,\ldots,n-1$, we construct the rational function
\begin{equation}\label{main.rat.func.2}
\sum\limits_{k=1}^{n-1}\dfrac{|u_{k1}|^2}{\lambda-\mu_k}:=\dfrac{\psi(\lambda)}{\chi_{n-1}(\lambda)},
\end{equation}
where the polynomial $\psi(\lambda)$ of degree $n-2$ is uniquely
determined by the numbers $|u_{k1}|^2$ and $\mu_k$,
$k=1,\ldots,n-1$. Moreover, since $|u_{k1}|^2>0$ by construction,
the function $\psi(\lambda)/\chi_{n-1}(\lambda)$ maps the upper
half-plane to the lower half-plane of the complex plane, so the
zeroes of $\psi$ are real and simple and interlace the zeroes of
$\chi_{n-1}$.

If we know the function~\eqref{main.rat.func.2}, we can always reconstruct
a unique Jacobi matrix of the form
\begin{equation}\label{Jacobi.matrix.1}
\begin{pmatrix}
\widehat{c}_1 & \widehat{b}_1 &  0 &\dots&   0   & 0 \\
\widehat{b}_1 & \widehat{c}_2 &\widehat{b}_2 &\dots&   0   & 0 \\
0  &\widehat{b}_2 & \widehat{c}_3 &\dots&   0   & 0 \\
\vdots&\vdots&\vdots&\ddots&\vdots&\vdots\\
0  &  0  &  0  &\dots&\widehat{c}_{n-2}&\widehat{b}_{n-2}\\
0  &  0  &  0  &\dots&\widehat{b}_{n-2}&\widehat{c}_{n-1}\\
\end{pmatrix}=\widehat{J}_{n-1},
\end{equation}
where $\widehat{b}_1,\ldots,\widehat{b}_{n-1}>0$, $\widehat
c_1,\ldots,\widehat c_{n-1}\in\mathbb{R}$. There exist a few
algorithms to make such a
reconstruction~\cite{hochstadt,hald,G,Simon}.

Thus, given the numbers $\mu_k$ and $|u_{k1}|^2$ the rational
function~\eqref{main.rat.func.2} uniquely determines the matrix
$J_{n-1}$ whose eigenvalues are $\mu_k$, while
$\psi(\lambda)=\chi_{n-2}(\lambda)$. Furthermore, we put
\begin{equation*}
\widehat
a_n:=\left(\sum\limits_{j=1}^{n}\lambda_j-\sum\limits_{k=1}^{n-1}\mu_k\right)
\end{equation*}
\begin{equation*}
\widehat{b}_n:=|\widehat{b}_n|e^{i\arg\beta}
\end{equation*}
and
\begin{equation*}
\widehat{b}_{n-1}:=\dfrac{\beta}{\widehat{b}_1\cdots \widehat{b}_{n-2}\widehat{b}_n}.
\end{equation*}
So we constructed a matrix of the form~\eqref{main.matrix.subclass} such
that $\sigma(\widehat{J}_{n-1})=\{\mu_1,\ldots,\mu_{n-1}\}$ and $\widehat{b}_1\cdots
\widehat{b}_n=\beta$. To finish the proof it suffices to establish that
$\sigma(\widehat{J}_{n})=\{\lambda_1,\ldots,\lambda_{n}\}$. To do this we
consider the following formula obtained from~\eqref{main.rat.func.0}--\eqref{submatrix.resolvent.0}
\begin{equation}\label{inverse.odd.rational.f}
\dfrac{\det(\lambda I_n-\widehat{J}_n)}{\det(\lambda
I_{n-1}-\widehat{J}_{n-1})}=\lambda+\widehat a_n-
\sum\limits_{k=1}^{n-1}\dfrac{|\widehat{b}_nu_{k1}+\overline{\widehat{b}}_{n-1}u_{k,n-1}|^2}{\lambda-\mu_k},
\end{equation}
where $u_{k1}$ and $u_{k,n-1}$ are the first and the last entries of the orthonormal eigenvector
$\mathbf{u}_k$ of the submatrix $\widehat{J}_{n-1}$ corresponding to the eigenvalue $\mu_k$, $k=1,\ldots,n-1$. These entries are related as
in~\eqref{formula.vectors},
and $u_{k1}$ satisfies~\eqref{u_k1.formula} by construction. Then the numbers
$\widehat{b}_{n-1}$ and $|u_{k,n-1}|^2$, $k=1,\ldots,n-1$, must satisfy the
equalities~\eqref{abs.b_n-1.formula}--\eqref{u_k,n-1.formula}. Thus
$|\widehat{b}_nu_{k1}|^2$ and $|\overline{\widehat{b}}_{n-1}u_{k,n-1}|^2$
are the solutions of the equation~\eqref{Theorem.spectrum.necessary.condition.proof.equation1}
so we have
\begin{equation*}
|\widehat{b}_nu_{k1}|^2+|\overline{\widehat{b}}_{n-1}u_{k,n-1}|^2=-\dfrac{\chi_{n}(\mu_k)+2\Re\beta}{\chi'_{n-1}(\mu_k)}.
\end{equation*}
This formula together with~\eqref{formula.vectors} gives us
\begin{equation*}
\begin{array}{c}
|\widehat{b}_nu_{k1}+\overline{\widehat{b}}_{n-1}u_{k,n-1}|^2=|\widehat{b}_nu_{k1}|^2+
|\overline{\widehat{b}}_{n-1}u_{k,n-1}|^2+2\Re(\widehat{b}_nu_{k1}\widehat{b}_{n-1}\overline{u}_{k,n-1})=\\
\\
-\dfrac{\chi_{n}(\mu_k)+2\Re\beta}{\chi'_{n-1}(\mu_k)}+\dfrac{2\Re\beta}{\chi'_{n-1}(\mu_k)}=
-\dfrac{\chi_{n}(\mu_k)}{\chi'_{n-1}(\mu_k)}=-\dfrac{\prod\limits_{j=1}^{n}(\mu_k-\lambda_j)}{\prod\limits_{\substack
{r=1\\r\neq k}}^{n-1}(\mu_k-\mu_r)}
\end{array}
\end{equation*}
so $\det(\lambda I_n-\widehat{J}_n)=\prod\limits_{j=1}^{n}(\lambda-\lambda_j)$, as required.
\end{proof}
\begin{theorem}\label{Theorem.inverse.beta.nonreal.a_n.real}
Let $\beta$ be a nonreal number, and let $\{\mu_k\}_{k=1}^{n-1}$ and
$\{\lambda_j\}_{j=1}^{n}$ be sets of distinct real numbers with no common
elements.

Then there exists a matrix $\widehat{J}_n$ in the class
$\widehat{{\cal J}}_n$ with $\widehat{a}_n\in\mathbb{R}$ and
$\beta=\widehat{b}_1\cdots \widehat{b}_n$ such that
$\sigma(\widehat{J}_n)=\{\lambda_1,\ldots,\lambda_n\}$ and
$\sigma(\widehat{J}_{n-1})=\{\mu_1,\ldots,\mu_{n-1}\}$ if and only
if the conditions~\eqref{solvability.condition.beta.a_n.nonreal}
hold.
%
%
%

Moreover, if there are $m$ $(0\leqslant m\leqslant n-1)$ equalities
in~\eqref{solvability.condition.beta.a_n.nonreal}
, then there exist exactly $2^{n-m-1}$ matrices in the class
$\widehat{{\cal J}}_n$ with these spectral data.
\end{theorem}
%
%

For real nonzero $\beta$ we have the following
results.

\begin{theorem}\label{Theorem.inverse.beta.a_n.real}
Let $\beta$ be a real nonzero number, and let $\{\mu_k\}_{k=1}^{n-1}$ be a set of
real distinct numbers. Suppose that $\{\lambda_j\}_{j=1}^{n}$  is a set of complex
numbers in the \textbf{closed} upper half-plane of the complex plane, $\Im\lambda_j\geqslant0$, $j=1,\ldots,n-1$.

Then there exists a matrix $\widehat{J}_n$ in the class $\widehat{{\cal J}}_n$
with $\Im \widehat{a}_n>0$ and $\beta=\widehat{b}_1\cdots \widehat{b}_n$ such that
$\sigma(\widehat{J}_n)=\{\lambda_1,\ldots,\lambda_n\}$
and $\sigma(\widehat{J}_{n-1})=\{\mu_1,\ldots,\mu_{n-1}\}$ if and only if
\begin{equation}\label{solvability.condition.beta.a_n.real.1}
(-1)^{n-k}\chi_n(\mu_k)\geqslant0,\quad k=1,\ldots,n-1,
\end{equation}
and
\begin{equation}\label{solvability.condition.beta.a_n.real.2}
|\chi_n(\mu_k)|\geqslant4(-1)^{n-k-1}\beta,\quad k=1,\ldots,n-1,
\end{equation}
where $\chi_n(\lambda)=\prod\limits_{j=1}^{n}(\lambda-\lambda_j)$.

Moreover, if there are exactly $m_1$ equalities
in~\eqref{solvability.condition.beta.a_n.real.1} and exactly $m_2$
equalities in~\eqref{solvability.condition.beta.a_n.real.2}
$(0\leqslant m_1+m_2\leqslant n-1)$, then there exist exactly
$2^{n-m_1-m_2-1}$ matrices in the class $\widehat{{\cal J}}_n$ with
these spectral data.
\end{theorem}
\begin{theorem}\label{Theorem.inverse.beta.real.a_n.nonreal}
Let $\beta$ be a real number, and let $\{\mu_k\}_{k=1}^{n-1}$ and
$\{\lambda_j\}_{j=1}^{n}$ be sets of distinct real numbers that may have
common elements.

Then there exists a matrix $\widehat{J}_n$ in the class $\widehat{{\cal J}}_n$
with $\widehat{a}_n\in\mathbb{R}$ and $\beta=\widehat{b}_1\cdots \widehat{b}_n$ such that
$\sigma(\widehat{J}_n)=\{\lambda_1,\ldots,\lambda_n\}$
and $\sigma(\widehat{J}_{n-1})=\{\mu_1,\ldots,\mu_{n-1}\}$ if and only if
the numbers $\{\mu_k\}_{k=1}^{n-1}$ and $\{\lambda_j\}_{j=1}^{n}$ satisfy
the inequalities~\eqref{solvability.condition.beta.a_n.real.1}--\eqref{solvability.condition.beta.a_n.real.2}.

Moreover, if there are exactly $m_1$ equalities
in~\eqref{solvability.condition.beta.a_n.real.1} and exactly $m_2$
equalities in~\eqref{solvability.condition.beta.a_n.real.2}
$(0\leqslant m_1+m_2\leqslant n-1)$, then there exist exactly
$2^{n-m_1-m_2-1}$ matrices in the class $\widehat{{\cal J}}_n$ with
these spectral data.
\end{theorem}

Theorems~\ref{Theorem.inverse.beta.nonreal.a_n.real}--\ref{Theorem.inverse.beta.real.a_n.nonreal}
can be proved in the same way as we proved
Theorem~\ref{Theorem.inverse.beta.a_n.nonreal}. Note that the
numbers $m_1$ and $m_2$ do not exceed a half of $n$ (approximately).
The exact upper bounds on $m_1$ and $m_2$ depend on the parity of
$n$ and on the sign of $\beta$, and can be obtained from the
inequalities~\eqref{solvability.condition.beta.a_n.real.1}--\eqref{solvability.condition.beta.a_n.real.2}.

Theorem~\ref{Theorem.inverse.beta.real.a_n.nonreal} was established
in~\cite[Theorems~6--7]{XJO}. As well as in~\cite[Corollary~8]{XJO}
we note that in
Theorems~\ref{Theorem.inverse.beta.a_n.nonreal}--\ref{Theorem.inverse.beta.real.a_n.nonreal}
the constructed matrix $\widehat{J}_n$ is unique if and only if
$m=n-1$ or $m_1+m_2=n-1$. This is possible for specific
distributions of the $\lambda_j$ and the $\mu_k$.

\section{Conclusions}
In this paper, we showed that the technique developed by Y.-H.\,Xu
and E.-X.\,Jiang \cite{XJO} for periodic Jacobi matrices can be
extended to a class of complex band matrices. We also gave some new
explanations of the technique using the theory of the location of
roots of polynomials. Our results extend the results of Arlinskii
and Tsekhanovskii~\cite{Arlinskii_Tsekhanovskii} (see
also~\cite{Arlinskii_Tsekhanovskii2}) who studied a one-dimensional
imaginary perturbation of symmetric real Jacobi matrices.

We believe that the combination of the methods of the
works~\cite{XJO} and~\cite{Arlinskii_Tsekhanovskii2} can be helpful
to study direct and inverse spectral problems of some other classes
of band matrices.

\end{document}